\documentclass[10pt,bezier]{article}
\usepackage{amsmath,amssymb,amsfonts,graphicx,caption,subcaption}
\usepackage[colorlinks,linkcolor=red,citecolor=blue]{hyperref}

\newtheorem{prethm}{{\bf Theorem}}[section]
\newenvironment{thm}{\begin{prethm}{\hspace{-0.5
em}{\bf.}}}{\end{prethm}}
\newtheorem{prepro}{{\bf Theorem}}

\newtheorem{precor}[prethm]{{\bf Corollary}}
\newenvironment{cor}{\begin{precor}{\hspace{-0.5
em}{\bf.}}}{\end{precor}}
\newtheorem{preconj}[prethm]{{\bf Conjecture}}

\newtheorem{preremark}[prethm]{{\bf Remark}}
\newenvironment{remark}{\begin{preremark}\em{\hspace{-0.5
em}{\bf.}}}{\end{preremark}}
\newtheorem{prelem}[prethm]{{\bf Lemma}}
\newenvironment{lem}{\begin{prelem}{\hspace{-0.5
em}{\bf.}}}{\end{prelem}}
\newtheorem{preque}[prethm]{{\bf Question}}

\newtheorem{preobserv}[prethm]{{\bf Observation}}

\newtheorem{preproposition}[prethm]{{\bf Proposition}}

\newtheorem{preproof}{{\bf Proof.}}
\newtheorem{preprooff}{{\bf Proof}}

\newenvironment{proof}[1]{\begin{preproof}{\rm
#1}\hfill{$\Box$}}{\end{preproof}}

\newtheorem{preproofs}{{\bf Second proof of }}

\newtheorem{preprooft}{{\bf Third proof of }}

\newtheorem{preproofF}{{\bf Proofs of}}

\title{\bf\Large 
The existence of $\{\p,\q\}$-orientations in edge-connected graphs
}
\author{{\normalsize{\sc Morteza Hasanvand${}$} }\vspace{3mm}
\\{\footnotesize{${}$\it Department of Mathematical
 Sciences, Sharif
University of Technology, Tehran, Iran}}
{\footnotesize{}}\\{\footnotesize{   $\mathsf{morteza.hasanvand@alum.sharif.edu  }$ }}}
\date{}
\def\p  {p}
\def\q  {q}

\def\t {t}

\begin{document}
\maketitle
\begin{abstract}{ 
In 1976 Frank and Gy{\'a}rf{\'a}s gave a necessary and sufficient condition for the existence of an orientation in an arbitrary graph $G$ such that for each vertex $v$, the out-degree $d^+_G(v)$ of it satisfies  $\p(v)\le d^+_G(v)\le \q(v)$, where  $\p$ and $\q$ are two integer-valued functions on $V(G)$ with $\p\le \q$. In this paper, we give a sufficient edge-connectivity condition  for the existence of an orientation in $G$ such that  for each vertex $v$,  $d^+_G(v)\in \{\p(v),\q(v)\}$, provided that  for each vertex $v$, $\p(v)\le \frac{1}{2}d_G(v) \le \q(v)$,  $|\q(v)-\p(v)|\le k$, and there is $\t(v)\in \{\p(v),\q(v)\}$ in which $|E(G)|=\sum_{v\in V(G)}\t(v)$. This result is a generalization of a theorem due to Thomassen (2012) on the existence of modulo orientations in  highly edge-connected graphs.
\\
\\
\noindent {\small {\it Keywords}:
\\
Modulo orientation;
edge-connecticity;
out-degree;
spanning tree.

}} {\small
}
\end{abstract}
%
%
%
%
%
%
%
%
%
%
%
%
%
%
\section{Introduction}
In this article, graphs have  no  loops, but  multiple  edges are allowed, and a general graph 
may have loops and multiple edges.
Let $G$ be a  graph.  The vertex set and  the edge set of $G$ are denoted by $V(G)$ and $E(G)$, respectively.
We denote by $d_G(v)$ the degree of a vertex $v$ in the graph $G$.
If $G$ has an orientation,  the out-degree and  in-degree of $v$ are denoted by $d_G^+(v)$ and $d_G^-(v)$.
For a vertex set $A$ of $G$ with at least two vertices, the number of edges of $G$ with exactly one end in $A$ is denoted by $d_G(A)$.
Also, we denote by $e_G(A)$ the number of edges with both ends in $A$.
For each vertex, let $L(v)$ be a set of integers.
We denote by  ${\it gap}(L(v))$ the maximum of all $|a-b|$
 taken over of consecutive integers in 
$a,b\in L(v)$, and the denote by $gap(L)$ the maximum of all $gap(L(v))$ taken over of all vertices $v$.
An orientation  of $G$ is said to be 
(i)
{\it $L$-orientation}, if for each vertex $v$, $d^+_G(v)\in L(v)$,
(ii) {\it $z$-defective $L$-orientation}, if for each vertex $v$ with $v\neq z$, $d^+_G(v)\in L(v)$,
 (iii) {\it $(\p,\q)$-orientation}, if for each vertex $v$, $\p(v)\le d^+_G(v)\le \q(v)$,
where $\p$ and $\q$ are two integer-valued functions on $V(G)$.
Let $k$ be a positive integer. 
The cyclic group of order $k$ is denoted by  $Z_k$.
An orientation of $G$ is said to be 
{\it $p$-orientation}, if for each vertex $v$, $d_G^+(v)\stackrel{k}{\equiv}p(v) $,
 where $p:V(G)\rightarrow Z_k$ is a mapping.
A graph $G$ is called 
{\it $m$-tree-connected}, if it contains  $m$ edge-disjoint spanning trees. 
Note that by the result of Nash-Williams~\cite{Nash-Williams-1961} and Tutte~\cite{Tutte-1961}  every $2m$-edge-connected graph is $m$-tree-connected.
A  graph $G$ is  said  to be {\it $(m,l_0)$-partition-connected},
 if it can be decomposed into an $m$-tree-connected factor and a factor $F$ having an orientation such that for each $v$, $d^+_G(v)\ge l_0(v)$, where $l_0$  is a nonnegative integer-valued function on $V(G)$. 
For a graph $G$ with a vertex $z$, we denote by $\chi_z$  the mapping $\chi_z:V(G)\rightarrow \{0,1\}$
such that $\chi(z)=1$ and $\chi(v)=0$ for all vertices $v$ with $v\neq z$. Also, we define $\bar{\chi}_z=1-\chi_z$.
 For two edges $xu$ and $uy$ incident with the vertex $u$, {\it  lifting of $xu$ and $uy$}
 is an operation that removes $xu$ and $uy$ and adds a new edge $xy$
(when the purpose is to generate a loopless graph we must not add the next edge $xy$ when $x=y$).
Throughout this article,  all variables $k$ and $m$ are positive  integers.

In 1965  Hakimi introduced the following criterion for the existence of an orientation with a given upper bound on out-degrees.
\begin{thm}{\rm (\cite{Hakimi-1965})}\label{thm:Intro:Hakimi}
{Let $G$ be a graph and let $\q$ be an integer-valued function on $V(G)$.
Then $G$ has an orientation such that for each $v\in V(G)$, $d^+_G(v)\le \q(v)$, if and only if $e_G(S)\le \sum_{v\in S}\q(v)$ for all $S\subseteq V(G)$.
}\end{thm}

In 1976  Frank and Gy{\'a}rf{\'a}s generalized Hakimi's result to the following bounded out-degree version.
\begin{thm}{\rm (\cite{Frank-Gyrfas-1976})}\label{thm:Into:Frank-Gyrfas}
{Let $G$ be a graph and let $\p$ and $\q$ be two integer-valued functions on $V(G)$ with $\p\le \q$.
Then $G$ has an orientation such that for each $v\in V(G)$, $\p (v)\le d^+_G(v)\le \q(v)$, if and only if for all $S\subseteq V(G)$,
$$e_G(S)\le \min\{\sum_{v\in S}q(v),\,  \sum_{v\in S}(d_G(v)-\p(v))\}.$$ 
}\end{thm}

In 2012 Thomassen gave a  sufficient edge-connectivity condition for the existence of  modulo orientations as the following theorem.
\begin{thm}{\rm(\cite{Thomassen-2012})}\label{Intro:orientation:thm:Thomassen}
{Let $G$ be a $(2k^2+k)$-edge-connected graph and let $p:V(G)\rightarrow Z_k$ be a mapping.
Then $G$ has a $p$-orienatation if and only if $|E(G)| \stackrel{k}{\equiv} \sum_{v\in V(G)}p(v)$.
}\end{thm}

In this paper, we provide a development for Thomassen's result by giving a  sufficient edge-connectivity for the existence of $\{\p,\q\}$-orientations as the following theorem.
\begin{thm}\label{Into:thm:p,q-orientation}
{Let $G$ be a $8k^2$-edge-connected graph  and let $\p$ and $\q$ be two  integer-valued functions on $V(G)$  in which for each vertex $v$, $\p(v)\le d_G(v)/2\le \q(v)$ and $|\q(v)-\p(v)|\le k$, 
Then $G$ has an orientation such  that for each vertex $v$, $d^+_G(v)\in \{\p(v),\q(v)\}$
if and only if  there is an integer-valued function $t$ on $V(G)$ in which $\t(v)\in \{\p(v),\q(v)\}$  for each vertex $v$,  and 
$|E(G)|=\sum_{v\in V(G)}\t(v)$.
}\end{thm}
%
%
%
%
%
%
%
%
%
%
\section{Edge-connected graphs: $(\p,\q)$-orientations and $\{\p,\q\}$-orientations}
\subsection{$(\p,\q)$-orientations}
In this section, we are going to derive some corollary of the following reformulation of Hakimi's Theorem.
This version exhibits  that why edge-connectivity plays an important role for finding orientations whose out-degrees are far from   the half of the corresponding degrees in $G$.
\begin{thm}{\rm (Hakimi~\cite{Hakimi-1965})}\label{thm:Hakimi}
{Let $G$ be a graph and let $\q$ be an integer-valued function on $V(G)$.
Then $G$ has an orientation such that for all $v\in V(G)$, $d^+_G(v)\le \q(v)$, if and only if for all $S\subseteq V(G)$,
$$ \sum_{v\in S}(d_G(v)-2\q(v))\le d_G(S),$$ 
Furthermore, under this condition, $d^+_G(v)= \q(v)$ for all $v\in V(G)$, if and only if  $|E(G)|=\sum_{v\in V(G)}\q(v)$.
}\end{thm}
\begin{proof}
{Apply Theorem~\ref{thm:Intro:Hakimi} and the fact that $\sum_{v\in S}d_G(v)/2-d_G(S)/2=e_G(S)$ for every vertex set $S$.
}\end{proof}
Another immediate consequence of Theorem~\ref{thm:Hakimi} is given in the next corollary.
\begin{cor}\label{cor:changing-orientations}
{Let $G$ be a graph and let $\q$ be an integer-valued function on $V(G)$ satisfying $|E(G)|\le \sum_{v\in V(G)}q(v)$.
If $G$ is $\lambda$-edge-connected, then it admits an orientation such that for each vertex $v$, $d^+_G(v)\le \q(v)$,
 where $$\lambda=\sum_{v\in V(G)}\max\{0,d_G(v)-2\q(v)\}.$$
Furthermore, under this condition, $d^+_G(v)= \q(v)$ for all $v\in V(G)$, if and only if  $|E(G)|=\sum_{v\in V(G)}\q(v)$.
}\end{cor}
\begin{proof}
{For every nonempty proper subset $S$ of $V(G)$,
$\sum_{v\in S}(d_G(v)-2\q(v))\le \sum_{v\in V(G)}\max\{0,d_G(v)-2\q(v)\}=\lambda\le d_G(S)$.
If $S=V(G)$, then by the assumption, $ \sum_{v\in S}(d_G(v)-2\q(v))\le 0=d_G(S)$.
Now, it enough to apply Theorem~\ref{thm:Hakimi}.
}\end{proof}
The following corollary makes an interesting tool  for constructing orientations from a given orientation.
\begin{cor}\label{cor:epsilon}
{Let $G$ be a graph with an orientation $D$ and 
  let $\varepsilon$ be a rational number with $0\le  \varepsilon \le 1$. 
Then $G$ admits an orientation $D_0$ such that  for all $v\in V(G)$,
$d^+_{D_0}(v)=\frac{1-\varepsilon}{2}d_G(v)+\varepsilon d^+_D(v)$
if and only if  for all $v\in V(G)$, $\frac{1-\varepsilon}{2}d_G(v)+\varepsilon d^+_D(v)$ is integer.
}\end{cor}
\begin{proof}
{By Theorem~\ref{thm:Hakimi}, for every vertex $S$, we have
 $ d_G(S) \ge \sum_{v\in S}(d_G(v)-2d^+_D(v))  \ge
  \sum_{v\in S}(\varepsilon d_G(v)-2\varepsilon d^+_D(v)) =
  \sum_{v\in S}\varepsilon(d_G(v)-2f(v))$, where $f(v)=\frac{1-\varepsilon}{2} d_G(v)+\varepsilon d^+_D(v)$. If $f$ is  integer-valued, 
then by Theorem~\ref{thm:Hakimi}, one can deduce that there is an orientation $D_0$ such that  for all $v\in V(G)$,
$d^+_{D_0}(v)=f(v)$. This can complete the proof.
}\end{proof}
\begin{cor}
{Let $G$ be a graph and let $k$ and $k_0$ be two odd positive integers with $k_0\le k$.
If $G$ has an orientation  $D$  such that  for all $v\in V(G)$,
$d^+_D(v)-d_G(v)/2\in \{0,\pm k/2\}$, then it has an orientation $D_0$
 such that for all $v\in V(G)$, $d^+_{D_0}(v)-d_G(v)/2\in \{0,\pm k_0/2\}$.
}\end{cor}
\begin{proof}
{Apply Corollary~\ref{cor:epsilon} with $\varepsilon=k_0/k$. Note that
if $d^+_D(v)=d_G(v)/2$, then $(1-\varepsilon) d_G(v)/2+\varepsilon d^+_D(v)=d_G(v)/2$,
and if $d^+_{D}(v)=d_G(v)/2\pm k/2$, then $(1-\varepsilon) d_G(v)/2+\varepsilon d^+_D(v)=d_G(v)/2\pm k/2$.
}\end{proof}
The following theorem is an edge-connected  reformulation of Frank and Gy{\'a}rf{\'a}s' Theorem.
\begin{thm}{\rm (\cite{Frank-Gyrfas-1976})}
{Let $G$ be a graph and let $\p$ and $\q$ be two integer-valued functions on $V(G)$ with $\p\le \q$.
Then $G$ has an orientation such that for each vertex $v$, $\p (v)\le d^+_G(v)\le \q(v)$, if and only if 
 for all $S\subseteq V(G)$.
$$ \max\{\sum_{v\in S}(2\p(v) -d_G(v)),  \sum_{v\in S}(d_G(v)-2\q(v))\}\le d_G(S),$$ 
}\end{thm}
\begin{proof}
{Apply Theorem~\ref{thm:Into:Frank-Gyrfas} and use the fact that $d_G(S)=e_G(S)-\sum_{v\in S}d_G(v)/2$.
}\end{proof}
\subsection{Defective $\{\p,\q\}$-orientations}
In order to prove Theorem~\ref{Into:thm:p,q-orientation}, we shall first  formulate a weaker  version.
For this purpose, we need the following two lemmas.  The first one guarantees the existence of modulo orientations with bounded out-degrees in edge-connected graphs which  is a refinement of the main result in~\cite{Lovasz-Thomassen-Wu-Zhang-2013}.
\begin{lem}{\rm (\cite{ModuloBounded})}\label{lem:modulo:3k-3}
{Let $G$ be a graph, let $n$ be a positive integer, and let $p:V(G)\rightarrow Z_n$ be a mapping satisfying 
$|E(G)|\stackrel{n}{\equiv}\sum_{v\in V(G)}p(v)$.
If $G$ is $(3n-3)$-edge-connected, then it has a $p$-orientation such that for each vertex $v$,
$|d^+_G(v) -d_G(v)/2|<n$.
Furthermore, for an arbitrary vertex $z$, we can  have 
$-x\le d^+_G(z)-d_G(z)/2<n-x$, where $x$ is an arbitrary real number $x\in [0, n)$.
}\end{lem}
\begin{lem}{\rm (see \cite{Edmonds-Johnson-1973})}\label{lem:parityforest}
{Let $G$ be a connected graph with $Q\subseteq V(G)$. 
If $|Q|$ is even, then $G$ has a spanning forest $F$ such that $Q=\{v\in V(F): d_F(v) \text{ is odd}\}$.
}\end{lem}
The following theorem gives a sufficient edge-connectivity for the existence of defective $\{\p,\q\}$-orientations.
\begin{thm}\label{thm:z-defective-p,q}
{Let $G$ be a graph with $z\in V(G)$, let $k$ be a positive integer, and let $\p$ and $\q$ be two  integer-valued functions on $V(G)$  in which for each vertex $v$, $\p(v)\le d_G(v)/2\le \q(v)$ and $|\q(v)-\p(v)|\le k$.
If $G$ is $(\frac{3}{2}k+1)(k-1)$-tree-connected, then it has an orientation such  that for each $v\in V(G)\setminus \{z\}$,
$$d^+_G(v)\in \{\p(v),\q(v)\}.$$
Furthermore, for the vertex $z$, we can  have 
$-x\le  d^+_G(z)-d_G(z)/2<k-x$, where $x$ is an arbitrary real number $x\in [0, k)$.
}\end{thm}
\begin{proof}
{We may assume that $k\ge 2$, as the assertion trivially holds when $k=1$. 
Since $G$ is $m$-tree-connected, we  can decompose $G$ into
  $k-1$ spanning trees $T_2,\ldots, T_k$ and
 $k-1$ factors $H_2,\ldots, H_k$ such that every $H_i$ is $(3i-3)$-tree-connected, where 
$m=\sum_{2\le i\le k}(3i-3)+k-1$. 
For each  $i\in \{2,\ldots, k\}$, define
 $$V_i=\{v\in V(G)\setminus \{z\}:|\p(v)-\q(v)|=i\},$$
and  $U_i=V(G)\setminus (V_i\cup \{z\})$.
 In addition,  by Lemma~\ref{lem:parityforest}, we can take $F_i$ to be a spanning forest of $T_i$ such that
 for each $v\in U_i$, 
$d_{F_i}(v)+d_{H_i}(v)$ is even and
 for each  $v\in V_i$, $d_{F_i}(v)+d_{H_i}(v)$ and $d_{G}(v)$  have the same parity.
Note that the parity of $d_{F_i}(z)$ can be determined by the degree of the other vertices.
Define $G_i=F_i\cup H_i$.
Finally, define $G_1=G\setminus E(G_2\cup \cdots \cup G_k)$, 
 $V_1=\{v\in V(G)\setminus \{z\}:|\p(v)-\q(v)|\le 1\}$, and  $U_1=\emptyset$.
According to the construction,  
 for each  $v\in U_i$,
$d_{G_i}(v)$ is even, and
for each  $v\in  V_i$,
$d_{G_i}(v)$ and $d_{G}(v)$ have the same parity,  whether   $2\le  i\le k$ or $i=1$.
By applying  Lemma~\ref{lem:modulo:3k-3}, for each $n\in \{1,\ldots, k\}$, 
we can recursively consider an orientation for $G_n$ such that   
 for each $v\in U_n$, $d^+_{G_n}(v)=d_{G_n}(v)/2$,  and for each $v\in V_n$,
$$d^+_{G_n}(v)-d_{G_n}(v)/2\in \{\p(v)-d_G(v)/2, \q(v)-d_G(v)/2\}, $$
and also $-x_n\le d^+_{G_n}(z)-d_{G_s}(z)/2<k-x_n$,  
where
$$x_n=\sum_{1\le j< n}(d^+_{G_j}(z)- d_{G_j}(z)/2)+x.$$
Note that  $0\le x_{n+1} <k$, since $ x_{n+1}=x_n+d^+_{G_n}(z)-d_{G_n}(z)/2$.
Let $v\in V(G)\setminus \{z\}$ so that  $v\in V_i$ and $1\le i \le k$.
Therefore,
$$d^+_{G}(v)=\sum_{1\le j\le k}d^+_{G_j}(v)=
d^+_{G_i}(v)+\sum_{1\le j\le k, j\neq i}d_{G_j}(v)/2=
d^+_{G_i}(v)+d_{G}(v)/2-d_{G_i}(v)/2\in \{\p(v), \q(v)\}.$$ 
Furthermore
$-x\le  d^+_G(z)-d_G(z)/2<k-x$, since $0\le x_{k+1} <k$.
Hence the proof is completed.
}\end{proof}
%
%
%
\subsection{$\{\p,\q\}$-orientations}
In this section, we shall improve Theorem~\ref{thm:z-defective-p,q} by refining the condition for the  vertex $z$.
To do this, we first form the following lemma for working with integer numbers.
Note that the following upper bound of $k(k-1)$ is sharp by setting  $(m,n)=(k-1,k)$, $x_i=k$, and $y_j=k-1$, where $1\le i\le m$ and $1\le j\le n$.
\begin{lem}\label{lem:additive}
{Let  $x_1,\ldots,  x_m$ and  $y_1,\ldots,  y_{n}$ 
be  positive integers and let $k$ be the maximum of them.
If $\sum_{1\le i\le m}x_i=\sum_{1\le j\le n}y_j$ 
and for any two  integer sets $I\subseteq \{1,\ldots, m\}$ 
and $J \subseteq \{1,\ldots, n\}$
 satisfying $|I|+|J|<m+n$, $\sum_{i\in I}x_i\neq  \sum_{j\in J}y_j$,
then $$\sum_{1\le i\le m}x_i\le  k(k-1).$$
}\end{lem}
\begin{proof}
{We may assume that $k$ is the maximum of $y_1,\ldots,  y_{n}$ and so $\max_{1\le i\le m}x_i < k$.
Let $I_0=J_0=\emptyset$ and $g(0)=f(0)=0$.
Let $s$ be a positive integer. 
If $I_1\cup \cdots \cup I_{s-1}\neq \{1,\ldots, m\}$, 
then we recursively define $I_s$ to be a nonempty subset of 
$\{1,\ldots, m\}\setminus (I_1\cup \cdots \cup I_{s-1})$
such that 
$$f(s-1) = \sum_{1\le t \le s-1}\sum_{j\in J_t}y_j\le \sum_{1\le t \le s}\sum_{i\in I_t}x_i= g(s).$$
If $J_1\cup \cdots \cup J_{s-1}\neq \{1,\ldots, n\}$, 
then we recursively define $J_s$ to be a nonempty subset of $\{1,\ldots, n\}\setminus (J_1\cup \cdots \cup J_{s-1})$
such that 
$$g(s)=\sum_{1\le t \le s}\sum_{i\in I_t}x_i\le \sum_{1\le t \le s}\sum_{j\in J_t}y_j=f(s).$$
%
We consider $I_s$ and $J_s$ with the minimum size.
These can imply that $g(1)-f(0)=\min_{1\le i\le m}x_i \le  k-1$ and 
$g(s)-f(s-1)\le \min_{i\in I_s} x_i-1\le k-1$ when $s>1$ and $f(s)-g(s)\le \min_{j\in J_s} y_j-1\le k-1$.
According to the assumption on  summations of $x_i$ and $y_j$, we must also have $f(s-1)\neq g(s)$  and $g(s)\neq f(s)$.
Define $g_d(s)=g(s)-f(s-1)$ and $f_d(s)=f(s)-g(s)$.
Assume that  $I_1\cup  \cdots \cup I_q=\{1,\ldots, m\}$ and $J_1\cup \cdots \cup J_q=\{1,\ldots, n\}$ so that
$$0=g(0)=f(0)<g(1)< f(1)< \cdots <f(q-1)\le  g(q)=\sum_{1\le i\le m}x_i= \sum_{1\le j\le n}y_j=f(q).$$
Let $s,s'\in \{1,\ldots, q\}$ with $s\ge s'$.
If $f_d(s)=f_d(s')$, then 
$$\sum_{s'< t\le s}\sum_{i\in I_t}x_i=g(s)-g(s')=f(s)-f(s')=\sum_{s'< t\le s}\sum_{j\in J_t}y_j.$$
According to the assumption on  summations of $x_i$ and $y_j$, we must have $s=s'$.
Note that we consider $J_q$ to be the empty, and consider  $I_q$ to be the empty set when $g(q)=f(q-1)$.
Similarly, if $g_d(s)=g_d(s')$, then $s=s'$.
In other words, $g_d$ and $f_d$  are injective  functions with the restricted domain $\{1,\ldots, q\}$ and the co-domain 
$\{0,1,\ldots, k-1\}$.
Therefore,
$$\sum_{1\le i\le m}x_i=g(q)=
\sum_{1\le s\le q}g_d(s)+\sum_{1\le  s\le  q-1}f_d(s)
 \le 2\sum_{1\le i\le k-1}i=k(k-1).$$
Hence the proof is completed.
}\end{proof}
The following theorem gives a sufficient edge-connectivity for the existence of $\{\p,\q\}$-orientations..
\begin{thm}\label{thm:exact-p,q}
{Let $G$ be a $4k^2$-tree-connected graph  and let $\p$ and $\q$ be two integer-valued functions on $V(G)$  in which for each $v\in V(G)$, $\p (v)\le d_G(v)/2\le\q(v)$ and $|\q(v)-\p(v)|\le k$.
Then $G$ has an orientation such  that  for each $v\in V(G)$,
$$d^+_G(v)\in \{\p(v),\q(v)\},$$
if and only if  there is an integer-valued function $t$ on $V(G)$ in which $\t(v)\in \{\p(v),\q(v)\}$ for  each $v\in V(G)$,  and 
$|E(G)|=\sum_{v\in V(G)}\t(v)$.
Furhermore, for an arbitrary given  vertex $z$, we can have $d^+_G(z)=\t(z)$.
}\end{thm}
\begin{proof}
{Since every $2$-tree-connected graph has a spanning Eulerian subgraph~\cite{Jaeger-1979}, 
one can decompose $G$ into a $2k^2$-tree-connected graph $G_0$ and a $2k^2$-edge-connected Eulerian graph $H$.
By Theorem~\ref{thm:z-defective-p,q},
 the graph $G_0$ has an orientation such  that
for each $v\in V(G)\setminus \{z\}$, $d^+_{G_0}(v)\in \{\p(v)-d_{H}(v)/2,\q(v)-d_{H}(v)/2\}$, 
and $|d^+_{G_0}(z)-d_{G_0}(z)/2|\le k$ in which
$d^+_{G_0}(z)\ge d_{G_0}(z)/2$ if and only  if $t(z)\ge d_G(z)/2$.
For each vertex $v$, define $s(v)=\t(v)-d^+_{G_0}(v)-d_{H}(v)/2$. 
According to this definition, for each $v\in V(G)\setminus \{z\}$,
$s(v)=0$ when $d^+_{G_0}(v)=t(v)-d_H(v)/2$, and $|s(v)| =|\q(v)-\p(v)|\le k$ otherwise.
 In addition, $|s(z)|=|\t(z)-d_{G}(z)/2-(d^+_{G_0}(z)-d_{G_0}(z)/2) |\le k$. 
Let $S$ be a subset of $V(G)$ including $z$  satisfying 
$\sum_{v\in S}s(v)=0$.
Note that $V(G)$ is a candidate for $S$, since
$$\sum_{v\in V(G)}s(v)=\sum_{v\in V(G)}(\t(v)-d^+_{G_0}(v)-d_{H}(v)/2)=|E(G)|-|E(G_0)|-|E(H)|=0.$$
 Consider $S$ with the minimum $|S|$. 
Thus  for every nonempty proper subset $S_0$ of $S$, $\sum_{v\in S_0}s(v)\neq 0$. Otherwise, 
 $\sum_{v\in S\setminus S_0}s(v)= 0$ which is a contradiction,  because  either $S_0$ or $S\setminus S_0$ includes $z$.
Thus by Lemma~\ref{lem:additive} and the minimal property of $S$,  
one can conclude that $\sum_{v\in S}|s(v)|\le 2k(k-1)$.
More precisely, variables $x_i$ in  Lemma~\ref{lem:additive}  are those positive integers $|s(v)|$ with $s(v)> 0$ and 
  variables $y_j$ are those positive integers $|s(v)|$ with $s(v)< 0$, where $v\in S$. 
Since $H$ is $2k(k-1)$-edge-connected, by Corollary~\ref{cor:changing-orientations},
 it has an orientation such that for each $v\in S$, $d^+_{H}(v)=d_H(v)/2+s(v)$ and 
for each $v\in V(G)\setminus S$, $d^+_{H}(v)=d_H(v)/2$.
Note that $\sum_{v\in S}\max\{0, d_H(v)-2(d_H(v)/2+s(v))\}= \sum_{v\in S}|s(v)|\le 2k(k-1)$.
Consider the orientation of $G$ obtained from these orientations.
For each vertex $v$, 
$$d^+_{G}(v)= d^+_{G_0}(v)+d^+_{H}(v)=
 \begin{cases}
d^+_{G_0}(v)+d_H(v)/2+s(v)=\t(v)\in \{\p(v),\q(v)\},	&\text{if $v\in S$};\\ 
d^+_{G_0}(v)+d_{H}(v)/2\in \{\p(v),\q(v)\},	&\text{otherwise}.
\end {cases}$$
Hence the theorem holds.
}\end{proof}
\begin{remark}
{We will use the above-mentioned theorem to refine some results in \cite{AddarioBerry-Dalal-McDiarmid-Reed-Thomason-2007, AddarioBerry-Dalal-Reed-2008} for edge-connected graphs.
We will do it in a forthcoming paper.
}\end{remark}
\section{Partition-connected graphs: orientations with sparse lists on out-degrees}
In this subsection, we are going to prove the following assertion on the existence of orientations with  sparse lists on out-degrees in partition-connected graphs.  For dense  lists in all graphs, it was investigated by Akbari, Dalirrooyfard, Ehsani, Ozeki, and Sherkati (2020) \cite{ADEOS}.
Before stating the main result, we need to recall the following lemma from~\cite{ModuloBounded}.
\begin{lem}{\rm (\cite{ModuloBounded})}\label{lem:preserving}
{Let $G$ be a general  graph with $z\in V(G)$ and  let  $l_0$ be a nonnegative integer-valued function on $V(G)$.
Assume that $z$ is not incident with loops.
If $G$ contains an $(m, l_0)$-partition-connected factor $H$ with $d_G(z)\ge 2d_H(z)-2l_0(z)-2$, then there are 
$d_H(z)-l_0(z)-1$ pair of edges incident with $z$ such that by lifting  them   the resulting general graph $G_0$ with
 $V(G_0)=V(G)\setminus \{z\}$ is still $(m,l_0)$-partition-connected.
}\end{lem}
Now, are we are ready to prove the main result of this section.
\begin{thm}
{Let $G$ be a general graph with $z\in V(G)$ and let $L:V(G)\rightarrow 2^{\mathbb{Z}}$ be a mapping satisfying 
$gap(L)\le k$ and $gap(L(z))=k$.
Let $s$, $s_0$, and $l_0$  be  three integer-valued functions on $V(G)$ satisfying $s(v)+s_0(v)+gap(L(v))< d_G(v)$ and
$\max\{s(v),s_0(v)\}\le l_0(v)+(2k^2-gap(L(v))+1)\bar{\chi}_z(v)$ for each vertex $v$. 
If $G$ is $(2k^2, l_0)$-partition-connected, then it admits a
 $z$-defective $L$-orientation such that  for each vertex  $v$,
$$s(v) \le   d^+_G(v)   \le d_G(v)-s_0(v).$$
}\end{thm}
\begin{proof}
{We may assume that $l_0$ is nonnegative and $G$ is loopless. The proof is by induction on $|V(G)|$.
For $|V(G)|\le 2$ the proof is straightforward.
So, suppose $|V(G)|\ge 3$.
For notational simplicity, let us define $m=2k^2$.
For proving the theorem,  we shall consider the following four cases.
%
\vspace{3mm}
\\
{\bf Case 1. There is a vertex $u\in V(G)\setminus \{z\}$ with $d_G(u)=2l_0(u)+2m-r$ such that  
$0< r\le l_0(u)+m$ and 
$l_0(u)+m-i \in L(u)$, where $0 \le i \le \min \{r,gap(L(u))-1\}$.}

By Lemma~\ref{lem:preserving},
there are  $l_0(u)+m-r$  pair of edges incident with $u$ 
such that by lifting them the resulting general graph $H$ with $V(H)=V(G)\setminus u$
is still $(m, l_0)$-partition-connected.
Obviously, $d_R(u)=d_G(u)-2(l_0(u)+m-r)=r$,
where $R$ is the factor of $G$ consisting of all  edges incident with $u$ that are not lifted.
Since $i\le r$, the  edges of $R$ can be orientated
such that $ d^+_R(u)+l_0+m-r \in L(u)$.
Define $s'(u)=s(u)-(l_0(u)+m-r)$  and $s'_0(u)=s_0(u)-(l_0(u)+m-r)$.
By the assumption, we must have $\max\{s'(u),s'_0(u)\}\le d_R(u)-(gap(L(u))-1)$,
 and  $s'(u)+s'_0(u)\le d_R(u)-(gap(L(u))-1)$.
Therefore, if  $d_R(u)\ge  gap(L(u))-1$ then the orientation of $R$ can be selected such that
$s'(u)\le d^+_R(u)\le d_R(u)-s'_0(u)$.
If $d_R(u)\le gap(L(u))-1$, then we must automatically have
$$s'(u)\le 0\le d^+_R(u)\le d_R(u) \le d_R(u)-s'_0(u).$$
Define $L'(v)=\{j-d^+_R(v):j \in L(v)\}$, where $v\in V(H)$.
Obviously, $\max\{s(v)-d_{R}^+(v), s_0(v)-d_{R}^-(v) \}\le l_0(v)+m-(gap(L(v))-1)$ and 
$s(v)-d_{R}^+(v)+s_0(v)-d_{R}^-(v) +gap(L(v))-1 \le d_{H}(v)$.  
Thus by the induction hypothesis, 
$H$ has a $z$-defective $L'$-orientation  such that for each $v\in V(H)$, 
$$s(v)-d_{R}^+(v)\le d_{H}^+(v)\le d_{H}(v)-(s_0(v)-d_{R}^-(v))=d_{G}(v)-s_0(v)-d^+_{R}(v).$$
This orientation induces a $z$-defective $L$-orientation for $G$ such that for each  $v\in V(H)$,  
 $d^+_G(v)= d^+_{H}(v)+d^+_R(v)$, and also $d_{G}^+(u)= d^+_R(u)+l_0(u)+m-r.$
This can complete the  proof of Case 1. $\square$
%
%
%
\vspace{3mm}
\\
{\bf Case 2.   $d_G(z)<2l_0(z)+gap(L(z))-1$.}

Since $d_G(z)\ge  gap(L(z))-1$, we must have  $l_0(z)>0$ and hence there is an edge $zu$ incident with $z$
 such that the graph $G_0$ is $(m, l_0-\chi_z)$-partition-connected, where $G_0=G-zu$. 

First assume that  $s(z)<l_0(z)$.
Since $s(z)<l_0(z)$, we must have $s(z)\le l_0(z)-\chi_z(z)$.
Thus by the induction hypothesis, the graph $G_0$ has a $z$-defective $(L-\chi_u)$-orientation
 such that for each vertex $v$, $s(v)-\chi_u(v)\le d_{G_0}^+(v)\le  d_{G_0}(v)-(s_0(v)-\chi_z(v))$.
Now, this orientation induces the desired  $z$-defective $L$-orientation for $G$ by adding an edge directed from $u$ to $z$.

Now,  assume that  $s(z)=l_0(z)$.
This implies that $s_0(z)<l_0(z)$, because $s(z)+s_0(z)+gap(L(z))-1\le d_G(z)$ and $d_G(z)<2l_0(z)+gap(L(z))-1$.
Thus by the induction hypothesis, the graph $G_0$ has a  $z$-defective $(L-\chi_z)$-orientation such that for each vertex $v$,
$s(v)-\chi_z(v)\le d_{G_0}^+(v)\le  d_{G_0}(v)-(s_0(v)-\chi_u(v))$.
Now, this orientation induces the desired   $z$-defective $L$-orientation for $G$ by adding an edge directed from $z$ to $u$.
This completes the  proof of Case 2. $\square$

Now, by applying  Theorem~\ref{thm:z-defective-p,q}, 
the graph  $G$ has an orientation such that $|d^+_G(z)-d_G(z)/2|\le k/2$  and
for all $v\in V(G)\setminus \{z\}$, 
$d^+_G(z)\in \{\p(v),\q(v)\}$,
 where $\p(v)$ and $\q(v)$ are the integers in $L(v)$ with the smallest $|\q(v)-\p(v)|$ 
such that   $  \p(v)\le d_G(v)/2\le \q(v)$.
According to Case 2, $d_G(z)\ge 2l_0(z)+m-(k-1)$,
which  implies that $s_0(z)\le l_0(z)\le  d^+_G(z) \le d_G(z)-l_0(z)\le d_G(z)-s_0(z)$.
Let $v\in V(G)\setminus \{z\}$.
If $d_G(v)\ge 2l_0(v)+2m$, then we must  have 
$$s(v)\le \lfloor d_G(v)/2\rfloor-(gap(L(v))-1)\le  d^+_G(v)  \le 
\lceil d_G(v)/2\rceil+(gap(L(v))-1)\le  d_G(v)-s_0(v).$$
Otherwise,  $d_G(v)=2l_0(v)+2m-r$  in which  $0< r< gap(L(v))-1$.
According to Case 1, 
$\{l_0(v)+m-i: 0\le i \le r\} \cap L(v)=\emptyset ,$
which implies that
$$l_0(v)+m -gap(L(v))<\p(v)\le d^+_G(v) \le \q(v)<l_0(v)+m-r+gap(L(v))= d_G(v)-(l_0(v)+m -gap(L(v))),$$
and so $s(v)\le d^+_G(v)  \le  d_G(v)-s_0(v).$
Hence the proof is completed.
}\end{proof}
%
%
%
%
%
%
%
%
%
%
%
%
%
%
%
%
%
%
%


\begin{thebibliography}{10}
%
\bibitem{AddarioBerry-Dalal-McDiarmid-Reed-Thomason-2007}
 L.~Addario-Berry, K.~Dalal, C.~McDiarmid, B.A. Reed, and A.~Thomason, Vertex-colouring edge-weightings, Combinatorica 27 (2007)~1--12.

\bibitem{AddarioBerry-Dalal-Reed-2008}
L.~Addario-Berry, K.~Dalal, and B.A. Reed, Degree constrained subgraphs, Discrete Appl. Math. 156 (2008)~1168--1174.

\bibitem{ADEOS}
S. Akbari, M. Dalirrooyfard, K. Ehsani, K. Ozeki, and R. Sherkati, Orientations of graphs avoiding given lists on out-degrees, J. Graph Theory 93 (2020)~483--502.

\bibitem{Edmonds-Johnson-1973}
J. Edmonds and E.L. Johnson, Matching, Euler tours and the Chinese postman, Mathematical Programming 5 (1973) 88--124.

\bibitem{Frank-Gyrfas-1976}
 A.~Frank and A.~Gy{\'a}rf{\'a}s,  How to orient the edges of a graph? in Combinatorics, Coll Math Soc J Bolyai 18 (1976)~353--364.

\bibitem{Hakimi-1965}
 S.L. Hakimi, On the degrees of the vertices of a directed graph, J. Franklin Inst. 279 (1965)~290--308.

\bibitem{ModuloBounded}
M. Hasanvand, Modulo orientations with bounded out-degrees, arXiv:1702.07039.

\bibitem{Jaeger-1979}
 F.~Jaeger,  A note on sub-{E}ulerian graphs, J. Graph Theory 3  (1979)~91--93.

\bibitem{Nash-Williams-1961}
C.St.J.A.~Nash-Williams, Edge-disjoint spanning trees of finite graphs, J. London Math. Soc. 36 (1961)~445--450.

\bibitem{Lovasz-Thomassen-Wu-Zhang-2013}
 L.M. Lov{\'a}sz, C.~Thomassen, Y.~Wu, and C.-Q. Zhang, Nowhere-zero 3-flows and modulo {$k$}-orientations, J. Combin. Theory Ser. B 103 (2013)~587--598.

\bibitem{Thomassen-2012}
 C.~Thomassen, The weak 3-flow  conjecture and the weak circular flow conjecture, J. Combin. Theory Ser. B 102 (2012)~521--529.

\bibitem{Tutte-1961}
W.T.~Tutte, On the problem of decomposing a graph into $n$ connected factors, J. London Math. Soc. 36 (1961)~221--230.

\end{thebibliography}
\end{document}